\documentclass[12pt]{amsart}
\usepackage{amsmath, amsfonts, amssymb, amsthm}
\usepackage[dvipsnames]{xcolor}
\usepackage{tikz}
\usepackage{subcaption}
\usepackage{hyperref, cleveref}
\usepackage[left=3cm,right=3cm,top=3cm,bottom=3cm]{geometry}
\usepackage{enumitem}

\newtheorem{theorem}{Theorem}[section]

\newtheorem{conjecture}[theorem]{Conjecture}
\newtheorem{lemma}[theorem]{Lemma}

\everymath{\displaystyle}

\DeclareMathOperator{\rad}{rad}

\newcommand{\boxempty}{\mathbin{\square}}

\title{Improved Bounds for Burning Fence Graphs}

\author{Anthony Bonato}
\address{Department of Mathematics\\
Ryerson University\\
Toronto, ON\\
Canada, M5B 2K3} \email{abonato@ryerson.ca}
\author{Sean English}
\address{Department of Mathematics\\
University of Illinois Urbana-Champaign\\
Urbana, IL\\
USA 61801} \email{senglish@illinois.edu}
\author{Bill Kay}
\address{Oak Ridge National Laboratory\\
Oak Ridge, TN\\
USA 37831} \email{kaybw@ornl.gov}
\author{Daniel Moghbel}
\address{Department of Mathematics\\
Ryerson University\\
Toronto, ON\\
Canada, M5B 2K3} \email{dmoghbel@ryerson.ca}
\thanks{The authors gratefully acknowledge support from NSERC, Ryerson University, and the Fields Institute for Research in Mathematical Sciences}

\begin{document}

\begin{abstract}
Graph burning studies how fast a contagion, modeled as a set of fires, spreads in a graph. The burning process takes place in synchronous, discrete rounds. In each round, a fire breaks out at
a vertex, and the fire spreads to all vertices that are adjacent to a burning vertex. The burning number of a graph $G$ is the minimum number of rounds necessary for each vertex of $G$ to burn. We consider the burning number of the $m \times n$ Cartesian grid graphs, written $G_{m,n}$.\ For $m = \omega(\sqrt{n})$, the asymptotic value of the burning number of $G_{m,n}$ was determined, but only the growth rate of the burning number was investigated in the case $m = O(\sqrt{n})$, which we refer to as fence graphs. We provide new explicit bounds on the burning number of fence graphs $G_{c\sqrt{n},n}$, where $c > 0$.
\end{abstract}

\maketitle
	
\section{Introduction}

\textit{Graph burning} was introduced in \cite{Bon2014, Bon2016} to model the spread of social contagion, such as rumors, gossip, or memes in a network. Graph burning measures how fast contagion spreads in a given network, such as a social network like Facebook or Instagram. The smaller the burning number is, the faster a contagion spreads in the network.

Given a graph $G$, the burning process on $G$ is a discrete-time process defined as follows. Initially, at time $t=0$ all vertices are unburned. At each time step $t \geq 1$, one new unburned vertex
is chosen to burn (if such a vertex is available); such a vertex is called a \emph{source of fire}. If a vertex is burned, then it remains in that state until the end of the process. Once a vertex is
burned in round $t$, in round $t+1$ each of its unburned neighbors becomes burned. The process ends when all vertices of $G$ are burned (that is, let $T$ be the smallest positive integer such that
there is at least one vertex not burning in round $T-1$ and all vertices are burned in round $T$). The \emph{burning number} of a graph $G$, denoted by $b(G)$, is the minimum number of rounds needed
for the process to end. With our notation we have that $b(G) = T$. The vertices that are chosen to be burned are referred to as a \emph{burning sequence}. A shortest such sequence is called
\emph{optimal}. Note that optimal burning sequences have length $b(G).$ The problem of computing $b(G)$ is \textbf{NP}-complete in elementary graph classes such as spiders (that is, trees with exactly one vertex of degree greater than two) and path forests (which are disjoint unions of paths), as shown in \cite{Bessy2017}. Approximation algorithms for graph burning were studied in \cite{Bon-Kam,Kam}, and the burning of infinite Cartesian grids was studied in \cite{bgs}. A survey of graph burning may be found in \cite{bp}.

The burning number has been calculated for various graph families.  One particular family that will be useful to mention is paths, where in \cite{Bon2014}, it was shown that
$$
b(P_n)=\lceil \sqrt{n} \rceil.
$$
One of the most significant conjectures in the topic of graph burning states that the path is among the hardest connected graphs on $n$ vertices to burn. More precisely,
\begin{conjecture}[\cite{Bon2016}]
If $G$ is a connected graph of order $n$, then $b(G)\leq \lceil \sqrt{n}\rceil$.
\end{conjecture}

In the present paper, we focus on the burning number of certain grid graphs. The \textit{Cartesian product} of $G$ and $H$, denoted $G \boxempty H$, has vertex set $V(G) \times V(H)$ and $(u_1, u_2)(v_1, v_2) \in E(G \boxempty H)$ if either $u_1 = v_1$ and $u_2 v_2 \in E(H)$, or $u_1 v_1 \in E(G)$ and $u_2 = v_2$. We focus on the burning number of the $m \times n$ \emph{Cartesian grid} defined as $P_m \boxempty P_n$ and denoted by $G_{m,n}.$ The value of $b(G_{m,n})$ for $m$ a function of $n$ was first studied in \cite{Mits2017}.

\begin{theorem}[\cite{Mits2017}] \label{thm grid burning number}
For $m = m(n)$, \[b\big(G_{m,n}\big) = \begin{cases}
(1 + o(1))\sqrt[3]{\frac{3}{2} mn} \: & \mbox{ if }n \geq m = \omega\big(\sqrt{n}\big), \\[0.5cm]
\Theta\big(\sqrt{n}\big) \: & \mbox{ if } m = O\big(\sqrt{n}\big).
\end{cases} \]
\end{theorem}
Notice that while Theorem~\ref{thm grid burning number} gives an asymptotically tight value for the burning number of grids where $n \geq m = \omega\big(\sqrt{n}\big),$ only the growth rate is given in the remaining case where $m = O\big(\sqrt{n}\big)$. We refer to the family of grids $b(G_{c\sqrt{n},n})$ for constant $c > 0$ as \emph{fences}, as they are by definition wider than they are tall. \Cref{Fig: Burning 4 by 16 Fence} illustrates a burning sequence for the fence $G_{4,16}$.
\begin{figure}[h]
	\centering
		\begin{subfigure}[b]{0.99\textwidth}
		\centering
		\begin{tikzpicture}
		[scale = 0.99, every edge/.style = {draw = black, very thick}]
		\tikzstyle{vertex} = [circle, draw = black, fill = white]
		\tikzstyle{fire source} = [vertex, fill = Red!80]
		\tikzstyle{every node}=[font=\scriptsize]
		\draw (0,0) grid (15,3);
		\node[vertex] at (0,0) {4};
		\node[vertex] at (0,1) {3};
		\node[vertex] at (0,2) {4};
		\node[vertex] at (0,3) {5};
		\node[vertex] at (1,3) {4};
		\node[vertex] at (1,2) {3};
		\node[fire source] at (1,1) {2};
		\node[vertex] at (1,0) {3};
		\node[vertex] at (2,0) {4};
		\node[vertex] at (2,1) {3};
		\node[vertex] at (2,2) {4};
		\node[vertex] at (2,3) {5};
		\node[vertex] at (3,3) {6};
		\node[vertex] at (3,2) {5};
		\node[vertex] at (3,1) {4};
		\node[vertex] at (3,0) {5};
		\node[vertex] at (4,0) {6};
		\node[vertex] at (4,1) {5};
		\node[vertex] at (4,2) {4};
		\node[vertex] at (4,3) {5};
		\node[vertex] at (5,3) {4};
		\node[vertex] at (5,2) {3};
		\node[vertex] at (5,1) {4};
		\node[vertex] at (5,0) {5};
		\node[vertex] at (6,0) {4};
		\node[vertex] at (6,1) {3};
		\node[vertex] at (6,2) {2};
		\node[vertex] at (6,3) {3};
		\node[vertex] at (7,3) {2};
		\node[fire source] at (7,2) {1};
		\node[vertex] at (7,1) {2};
		\node[vertex] at (7,0) {3};
		\node[vertex] at (8,0) {4};
		\node[vertex] at (8,1) {3};
		\node[vertex] at (8,2) {2};
		\node[vertex] at (8,3) {3};
		\node[vertex] at (9,3) {4};
		\node[vertex] at (9,2) {3};
		\node[vertex] at (9,1) {4};
		\node[vertex] at (9,0) {5};
		\node[vertex] at (10,0) {6};
		\node[vertex] at (10,1) {5};
		\node[vertex] at (10,2) {4};
		\node[vertex] at (10,3) {5};
		\node[vertex] at (11,3) {6};
		\node[vertex] at (11,2) {5};
		\node[fire source] at (11,1) {4};
		\node[vertex] at (11,0) {5};
		\node[vertex] at (12,0) {6};
		\node[vertex] at (12,1) {5};
		\node[vertex] at (12,2) {6};
		\node[fire source] at (12,3) {5};
		\node[vertex] at (13,3) {6};
		\node[vertex] at (13,2) {5};
		\node[vertex] at (13,1) {4};
		\node[vertex] at (13,0) {5};
		\node[vertex] at (14,0) {4};
		\node[fire source] at (14,1) {3};
		\node[vertex] at (14,2) {4};
		\node[vertex] at (14,3) {5};
		\node[fire source] at  (15,3) {6};
		\node[vertex] at (15,2) {5};
		\node[vertex] at (15,1) {4};
		\node[vertex] at (15,0) {5};
		\end{tikzpicture}
		\end{subfigure}
\caption{A $4 \times 16$ fence, where a vertex is labeled $i$ if it is burned in round $i$. The burning sequence depicted here is $(x_1,x_2,x_3,x_4,x_5,x_6)$, where $x_i$ is the red vertex labeled $i$ for $1\le i \le 6$.}
\label{Fig: Burning 4 by 16 Fence}
\end{figure}
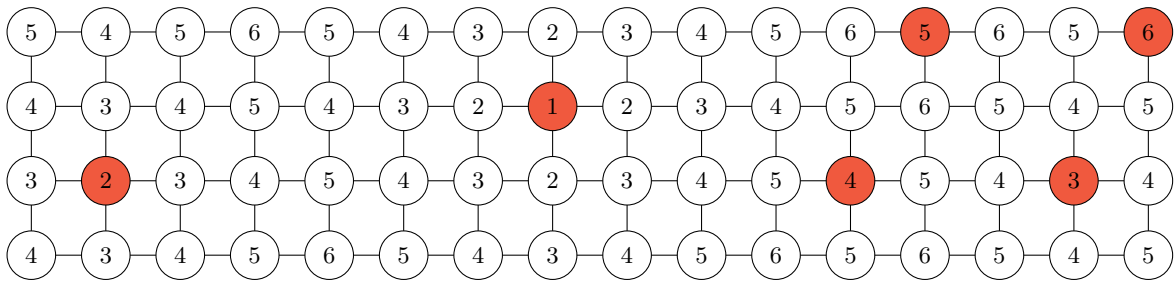

In the present paper, we improve on Theorem~\ref{thm grid burning number}, giving explicit lower and upper bounds on the burning number of fences. We prove the following theorem, which is our main result.
\begin{theorem} \label{thm main result}
Let $c > 0$. If $\ell = \max\{k\in\mathbb{N} : (k - 1)\sqrt{kn} + 1 \leq c\sqrt{n}\}$, then we have that
\[
b(c\sqrt{n},n) \geq \begin{cases}
\big(1 + o(1) \big) \left(\frac{c}2+\sqrt{1-\frac{c^2}4}\right) \sqrt{n} \,, & \mbox{ if } 0 < c < 2, \\[0.5cm]
\sqrt{\ell n} \,, & \mbox{ if } c \geq 2.
\end{cases}
\]
If $\ell=\lceil (c/2)^{2/3}\rceil,$ then we have that
\[
b(G_{c\sqrt{n},n})\leq 2\sqrt{\ell n}+\ell-1,
\]
and for $0<c\leq 2\sqrt{2}$, we have that
\[
b(G_{c\sqrt{n},n})\leq (1+o(1))\left(\frac{c}{2}+\sqrt{1-\frac{c^2}{16}}\right)\sqrt{n},
\]
\end{theorem}

The lower bound in Theorem \ref{thm main result} will follow immediately from Theorems~\ref{proposition c<2} and \ref{proposition c>2}, while the upper bound follows from Theorems~\ref{theorem upper bound large c} and \ref{proposition upper bound}. Potentially the most interesting case for fences is $G_{\sqrt{n},n}$. In this case, our lower bound is
\[
(1+o(1))\frac{1+\sqrt{3}}2\sqrt{n}\approx 1.366\sqrt{n},
\]
while our upper bound is
\[
(1+o(1))\frac{2+\sqrt{15}}4\sqrt{n}\approx 1.468\sqrt{n}.
\]
 It is worth noting here that the asymptotic value given in Theorem \ref{thm grid burning number} for $G_{m,n}$ when $m=\omega(\sqrt{n})$ does not hold when $m=O(\sqrt{n})$, as witnessed by the case $c=1$, in which the bound would give
 \[
 (1+o(1))\sqrt[3]{\frac{3}2}\sqrt{n}\approx 1.145\sqrt{n},
 \]
 far below our lower bound.

The rest of the paper is organized as follows: In Section \ref{subsection definitions}, we present definitions and notation we will use throughout the paper. In Section~\ref{subsection partial burning}, we present a useful tool in graph burning, which we refer to as the partial burning number. In Sections \ref{sec: lower bound} and \ref{sec: upper bound}, we will present a series of lemmas which will culminate in our lower and upper bounds, respectively.

\subsection{Definitions and Notation}\label{subsection definitions}

Denote the distance between vertices $u$ and $v$ of a graph $G$ by $d_G (u,v)$. The \emph{ball of radius $r$  centered at $v$} in $G$ is the set
\[
B_G(v,r) = \{w \in V(G) : d_G (v,w) \leq r\} .
\]
Often we will omit the subscript of $G$ and instead write $d(u,v)$ or $B(v,r)$ if it is clear from context which graph we are taking distances from. For a set $X$ of vertices and a vertex $u,$ we write $d(u,X)$ for the minimum distance between $u$ and a vertex of $X.$  Note that a source of fire $v$ that burns for $k$ rounds burns all the vertices in $B(v,k-1)$. The \emph{radius} of a graph $G,$ denoted $\mathrm{rad}(G),$ is the smallest integer $r$ such that there exists a vertex $v\in V(G)$ such that $B(v,r)=V(G)$. Given a graph $G$ and a vertex $u\not \in V(G)$, we let $G+u$ denote the graph with $V(G+u)=V(G)\cup\{u\}$ and $E(G+u)=E(G)$. Similarly, if $u,v\in V(G)$, $uv\not\in E(G)$, the graph $G+uv$ is the graph with vertex set $V(G+uv)=V(G)$ and edge set $E(G+uv)=E(G)\cup \{uv\}$. Given an integer $k$, we will let $[k]=\{x\in\mathbb{N}: x\leq k\}$. All graphs considered in this paper are finite, simple, and undirected. For more background on graph theory, the reader is directed to \cite{West}.


\subsection{Partial Burning}\label{subsection partial burning}

In burning a graph $G$, instead of requiring all the vertices to burn, we can instead only require that a subset $S \subseteq V(G)$ burns. We define the \textit{partial burning number of $G$ with respect to $S$} as the minimum number of rounds necessary to burn all the vertices in $S$ (and possibly some in $V(G)\setminus S)$). We denote this parameter by $b(G,S)$. Observe that $b(G,V(G)) = b(G)$.

In \cite{Bon2016}, it was observed that if $H$ is a spanning subgraph of $G$, $b(H)\leq b(G)$. Further, in \cite{thez}, it was noted that adding edges between distinct components of a graph can increase the burning number by at most one. We now present a result that generalizes and extends these ideas to the setting of partial burning.

\begin{lemma}\label{lemma subgraph burning}
	If $H$ is a subgraph of $G$ and $X\subseteq V(H)$, then
	\[
	b(G,X)\leq b(H,X)\leq b(G,X)+|E(G)|-|E(H)|.
	\]
\end{lemma}

\begin{proof}
By induction, it will suffice to show that
\begin{enumerate}[label=(\Alph*)]
	\item $b(H,X)\leq b(H+u+uv,X)+1$ for vertices $u\not \in V(H), v\in V(H)$.\label{item add a vertex}
	\item $b(H,X)\leq b(H+uv,X)+1$ for any edge $uv\in E(\overline{H})$.\label{item add an edge}
\end{enumerate}

For (\ref{item add a vertex}), we have that $b(H,X)=b(H+u+uv,X)$ since $u$ can only spread fire to $v$, and so any burning sequence containing $u$ can be replaced by a burning sequence containing $v$, and thus, will be a burning sequence in $H$. For (\ref{item add an edge}), fix some optimal burning sequence $S=(v_1,v_2,\dots,v_{b(H+uv,X)})$ of $H+uv$. Let us assume without loss of generality that when burning according to the sequence $S$, $u$ is burned in round $i$ and $v$ is burned in round $j$ with $i\leq j$ (in case one of $u$ or $v$ is not burnt after $b(H+uv,X)$ rounds, we will assume $v$ was not burnt). Observe that $(v,v_1,v_2,\dots,v_{b(H+uv,X)})$ is a burning sequence of $X$ in $H$ with length $b(H+uv,X)+1$ since the only effect the edge $uv$ could have had was to spread fire from $u$ to $v$, but here $v$ will already be burnt.
\end{proof}

\section{Lower Bound on the Burning Number of a Fence} \label{sec: lower bound}

Our lower bound for the burning number of fences will follow from analyzing the partial burning number of a collection of subpaths of $G_{m,n}$. If $P$ is a path in $G_{n,m}$, then we will say that $P$ is a \emph{horizontal path at height $h$} if $V(P)\subseteq \{v_{i,h}: 1\leq i\leq n\}$. Our next result will allow us to bound the number of vertices each source of fire burns if we are burning paths that are far enough apart.

\begin{lemma}\label{lemma conservation}
	Let $P^{(1)},P^{(2)},\dots,P^{(k)}\cong P_n$ be horizontal paths in $G=G_{m,n}$ at heights $h_1< h_2<\ldots<h_k,$ respectively, such that $h_i-h_{i-2}\geq 2t+2$ for $2 \leq i\leq k$. Let $\mathcal{P}=\bigcup_{i=1}^kV(P^{(i)})$. For $t\leq (n-1)/2,$ we have that
	\[
	\max_{v\in V(G)}|B(v,t)\cap \mathcal{P}|=\max_{v\in\mathcal{P}}|B(v,t)\cap \mathcal{P}|.
	\]
\end{lemma}

\begin{proof}
	Given a vertex $v\in V(G)$, note that by the bound on the distance between paths, $B(v,t)$ intersects at most two of the horizontal paths $P^{(i)}$, and so
	\begin{align*}
	|B(v,t)\cap \mathcal{P}|\leq\max\{2(t-d(v,V(P^{(a)})))+1,0\}+\max\{2(t-d(v,V(P^{(a+1)})))+1,0\},
	\end{align*}
	where $P^{(a)}$ and $P^{(a+1)}$ are the two horizontal paths closest to $v$. If the quantity $\max\{d(v,V(P^{(a)})), d(v,V(P^{(a+1)}))\}\geq t$, then $|B(v,t)\cap \mathcal{P}|\leq 2t+1\leq |B(v^*,t)\cap \mathcal{P}|$, where $v^*$ is a central vertex in $P^{(1)}$.  Otherwise,
	\begin{align*}
	|B(v,t)\cap \mathcal{P}|&\leq 4t-2(d(v,V(P^{(a)})))+d(v,V(P^{(a+1)})))+2\\
	&\leq 4t-2(d(V(P^{(a)}),V(P^{(a+1)})))+2\\
	&=|B(v^{**},t)\cap \mathcal{P}|,
	\end{align*}
	where $v^{**}$ is the closest vertex to $v$ in $P^{(a)}$, completing the proof.
\end{proof}

We have the following lemma.

\begin{lemma} \label{lemma burning far away paths}
Let $P^{(1)},P^{(2)},\dots,P^{(k)}\cong P_n$ be horizontal paths at heights $h_1< h_2<\ldots<h_k,$ respectively, in $G=G_{n,m}$ such that $h_i-h_{i-1}\geq \sqrt{kn}$ for $2 \leq i\leq k$. Let $\mathcal{P}=\bigcup_{i=1}^kV(P^{(i)})$. We then have that
\[
\sqrt{kn} \leq b(G, \mathcal{P}) \leq \sqrt{kn} + k - 1 .
\]
\end{lemma}

\begin{proof}
We first prove the lower bound. Note that since the horizontal paths $P^{(i)}$ are far apart, for all $1\leq t< \sqrt{kn}$,
	\[
	\max_{v\in\mathcal{P}}|B(v,t)\cap \mathcal{P}|\leq 2t+1.
	\]
	Furthermore, the conditions of Lemma~\ref{lemma conservation} are satisfied, so for all vertices $v\in V(G)$,
	\[
	|B(v,t)\cap \mathcal{P}|\leq 2t+1.
	\]
	Thus, if $S=(v_1,v_2,\dots,v_{\sqrt{kn}-1})$ is a burning sequence,
	\begin{align*}
	\sum_{i=1}^{\sqrt{kn}-1} |B(v_i,\sqrt{kn}-i-1)\cap \mathcal{P}|&\leq \sum_{i=0}^{\sqrt{kn}-2} 2i+1\\
	&=(\sqrt{kn}-1)^2\\
	&<kn=|\mathcal{P}|,
	\end{align*}
	and so $S$ cannot possibly burn all of $\mathcal{P}$, implying that $b(G, \mathcal{P})\geq \sqrt{kn}$.
	
	For the upper bound, observe that
	\begin{align*}
	b(G,\mathcal{P})&\leq b(G[\mathcal{P}])\\
	&=b(kP_n)\\
	&\leq b(P_{kn})+k-1=\sqrt{kn}+k-1,
	\end{align*}
	where the last inequality follows from Lemma~\ref{lemma subgraph burning} applied with $X=V(P_{k,n})$.
\end{proof}

To establish a lower bound on $b(G_{c\sqrt{n},n})$ for all $c > 0$, we consider the two cases separately: $0 < c < 2$ and $c \geq 2$. In the first case, our strategy for obtaining the lower bound is to consider the partial burning number of the top and bottom horizontal paths. In the second case, we will consider the partial burning number of a collection of horizontal paths that are sufficiently far apart.

\begin{theorem} \label{proposition c<2}
	If $0<c<2$ such that $c\sqrt{n}\in \mathbb{N}$, then
	\[
	b(G_{c\sqrt{n},n})\geq (1 + o(1))\left(\frac{c}2+\sqrt{1-\frac{c^2}4}\right) \sqrt{n}
	\]
\end{theorem}

\begin{proof}
	Let $P_{\bot}, P_{\top}\cong P_n$ be the horizontal paths at height $1$ and height $c\sqrt{n},$ respectively (that is, the bottom and top path), and let $\mathcal{P}=V(P{\bot})\cup V(P_{\top})$. We will bound the number $b(G_{c\sqrt{n},n},\mathcal{P})$.
	
	Let $(v_1,v_2,\dots,v_{b(G_{c\sqrt{n},n},\mathcal{P})})$ be an optimal burning sequence. For $0\leq t< c\sqrt{n}$, note that for vertices $v^*\in \mathcal{P}$, we have that $|B(v^*,t)\cap\mathcal{P}|\leq 2t+1$, so by Lemma \ref{lemma conservation}, $|B(v_{b(G_{c\sqrt{n},n})-t},t)\cap \mathcal{P}|\leq 2t+1$. This implies that the vertices in
	\[
	\{v_{b(G_{c\sqrt{n},n})-t}: 0\leq t< c\sqrt{n}-1\}
	\]
	can burn at most
	\begin{equation} \label{Covering sum 1}
	\sum_{i = 0}^{c \sqrt{n} - 1} (2i + 1)
	\end{equation}
	vertices in $\mathcal{P}$.
	
	Now, for $c\sqrt{n}\leq t\leq b(G_{c\sqrt{n},n})-1$, if $v^*\in \mathcal{P}$, then
	\[
	|B(v^*,t)\cap\mathcal{P}|\leq 2t+2+2(t-c\sqrt{n}-1),
	\]
	and so by Lemma~\ref{lemma conservation}, the vertices in
	\[
	\{v_{b(G_{c\sqrt{n},n})-t}: c\sqrt{n}\leq t< b(G_{c\sqrt{n},n})\}
	\]
	can burn at most
	\begin{equation} \label{Covering sum 2}
	\sum_{i = c \sqrt{n}}^{b(G_{c\sqrt{n},n})-1} \left(2i + 2 + 2\left(i - \left(c \sqrt{n} - 1 \right)\right)\right)
	\end{equation}
	vertices.
	
The sum of \eqref{Covering sum 1} and \eqref{Covering sum 2} must be greater than or equal to $2n = |\mathcal{P}|$. A computation in SageMath shows that such inequality holds if and only if \[b\big(G_{c\sqrt{n},n}\big) \geq \frac{1}{2} c\sqrt{n} \pm\sqrt{(c^2 + 4)n + 2 c^2 n - 4 c^2 n - 4 c^2 \sqrt{n} - 1} - \frac{3}{2} .\] Grouping the dominant terms, we derive that \[b\big(G_{c\sqrt{n},n}\big) \geq \left(\frac{1}{2} c + \frac{1}{2} \sqrt{4 - c^2} \right) \big(1 + o(1) \big) \sqrt{n} ,\]
and the proof follows.
\end{proof}

\begin{theorem} \label{proposition c>2}
If $c\geq 2$ such that $c\sqrt{n}\in\mathbb{N}$ and $\ell = \max\{k : (k - 1)\sqrt{kn} + 1 \leq c\sqrt{n}\}$, then
\[
b\big(G_{c\sqrt{n},n}\big) \geq \sqrt{\ell n}.
\]
\end{theorem}
		
\begin{proof}
	For $1\leq i\leq \ell$, let $P^{(i)}$ be the horizontal path in $G_{c\sqrt{n},n}$ at height $1+\sqrt{\ell n}(i-1)$, and let $\mathcal{P}=\cup_{i=1}^\ell V(P^{(i)})$. The conditions of Lemma \ref{lemma burning far away paths} are then satisfied with $k=\ell$, so we have that
		\begin{align*}
		b(G_{c\sqrt{n},n})\geq b(G_{c\sqrt{n},n},\mathcal{P})\geq \sqrt{\ell n},
		\end{align*}
and the proof follows.
\end{proof}

\section{Upper Bound on the Burning Number of a Fence} \label{sec: upper bound}

Before we focus on an upper bound for the burning number of fences, we will quickly show that determining the asymptotics for $b(G_{m,n})$ when $m=o(\sqrt{n})$ is trivial. We have the following result on the burning number of graph products.

\begin{theorem}[\cite{Mits2018}] \label{thm graph product max min bounds}
	If $G$ and $H$ are both connected graphs, then \[
	b(G \boxempty H) \leq \min\{b(G) + \rad (H), b(H) + \rad (G)\} .
	\]
\end{theorem}

Theorem~\ref{thm graph product max min bounds} gives the following.
\begin{lemma}
	If $m=o(\sqrt{n})$, then we have that
	\[
	b(G_{m,n})=(1+o(1))\sqrt{n}.
	\]
\end{lemma}

\begin{proof}
	For the lower bound, observe that to burn any horizontal path of order $n$ in $G=G_{m,n}$, it takes at least $\lceil\sqrt{n}\rceil$ rounds.
	
	For the upper bound, we apply Theorem~\ref{thm graph product max min bounds} to derive that
	\[
	b(G)\leq b(P_n)+\rad (P_m)=\sqrt{n}+o(\sqrt{n})=(1+o(1))\sqrt{n}.
	\]
The proof now follows.
\end{proof}

We return our attention to fences. To establish an upper bound, we first present a lemma which is a useful generalization of Theorem \ref{thm graph product max min bounds}.

\begin{lemma}\label{lemma product bound generalization}
If $X\subseteq V(G)$ is a subset of vertices such that each vertex $v\in V(G)$ is distance at most $k$ from $X,$ then we have that
\[	
b(G)\leq b(G,X)+k.
\]
\end{lemma}

\begin{proof}
	We first burn $X$ in $b(G,X)$ steps. Regardless of what other vertices we burn, after at most $k$ steps, the entire graph will be burned.
\end{proof}

In light of Lemma~\ref{lemma product bound generalization}, we can upper bound the burning number of our fence $G_{c\sqrt{n},n}$ if we can efficiently estimate $b\left(G_{c\sqrt{n},n},X\right)$ for some suitably chosen set $X\subseteq V(G)$. For our purposes, we will always choose $X$ to be the vertex set of a collection of horizontal paths. If $P^{(1)}$ and $P^{(2)}$ are two horizontal paths that are sufficiently far apart, then $b(G_{c\sqrt{n},n}, V(P^{(1)})\cup V(P^{(2)}))=2b(P_n)$. When $c$ is large, we find it useful to choose to burn few horizontal paths spaced evenly throughout the fence, but far apart from each other. In this case, since we do not need to worry about interactions between the horizontal paths we burn, we can easily provide an upper bound.

\begin{theorem}\label{theorem upper bound large c}
	If $\ell=\left\lceil (c/2)^{2/3}\right\rceil,$ then
	\[
	b(G_{c\sqrt{n},n})\leq 2\sqrt{\ell n}+\ell-1.
	\]
\end{theorem}

\begin{proof}
	For $0\leq i\leq \ell-1$, let $P^{(i)}$ denote the horizontal path in $G_{n,c\sqrt{n}}$ at height $(2\sqrt{\ell n}+1)i+\sqrt{\ell n}+1$. Let $\mathcal{P}=\bigcup_{i=0}^{\ell-1} V(P^{(i)})$. We then have that
	\begin{align*}
	b(G_{c\sqrt{n},n},\mathcal{P})&\leq b(\ell P_n)\\
	&\leq \sqrt{\ell n}+\ell-1,
	\end{align*}
	where the second inequality follows from Lemma \ref{lemma subgraph burning}, and the fact that $\ell P_n$ is a subgraph of $P_{\ell n}$ with $\ell-1$ edges missing.
	After each of these paths have burned, by our choice of $\ell$, we have that our $\ell$ horizontal paths are spaced exactly $2\sqrt{\ell n}$ apart, $P^{(0)}$ is at height $\sqrt{\ell n}+1$, and $P^{(\ell-1)}$ is at height
	\begin{align*}
	(2\ell-1)\sqrt{\ell n}+\ell&\geq 2\ell^{3/2}\sqrt{n}-\sqrt{\ell n}\\
	&\geq c\sqrt{n}-\sqrt{\ell n},
	\end{align*}
	so every vertex in $G_{c\sqrt{n},n}$ is distance at most $\sqrt{\ell n}$ from a burned vertex, so by Lemma \ref{lemma product bound generalization}, after at most $\sqrt{\ell n}$ more steps, the entire fence is burned.
\end{proof}

The preceding theorem holds for all values of $c$, but for small $c$, the upper bound can be improved by considering two horizontal paths, $P^{(1)}$ and $P^{(2)}$ that are close enough to each other such that $b(G_{c\sqrt{n},n}, V(P^{(1)})\cup V(P^{(2)}))<2b(P_n)$.

\begin{lemma} \label{lemma burn top and bottom}
Let $0 < c\leq \sqrt{2}$, and $G = G_{c\sqrt{n},n}$, let $P_{\bot}, P_{\top}\cong P_n$ be the horizontal paths in $G$ of order $n$ at height $1$ and $c\sqrt{n}$, respectively, and let $\mathcal{P}=V(P_\bot)\cup V(P_\top)$. We then have that
\[
b(G_{c\sqrt{n},n},\mathcal{P})\leq (1+o(1))\left(\frac{c}2+\sqrt{1-\frac{c^2}4}\right)\sqrt{n}
\]
\end{lemma}

\begin{proof}
Let $m=c\sqrt{n}$, $b=(1+o(1))\left(\frac{c}2+\sqrt{1-\frac{c^2}4}\right)\sqrt{n}$ and $V(G) = \{v_{i,j} : (i,j) \in [m] \times [n] \}$. Note that for all $i\geq 0$, $t\geq m-1$ we have that
\[
B(v_{1,i}, t) \cap B(v_{m,i+2t-m+1}, t - 1) = \emptyset,
\]
while the intersection of
\[
B(v_{1,i}, t) \cup B(v_{m,i+2t-m+1}, t - 1)
\]
with $V(P_\top)$ and with $V(P_\bot)$ induces a connected path in either case. Analogously, for all $i\geq 0$, $t\geq m-1$,
\[
B(v_{m,i}, t) \cap B(v_{1,i+2t-m+1}, t - 1) = \emptyset,
\]
while
\[
B(v_{m,i}, t) \cup B(v_{1,i+2t-m+1}, t - 1)
\]
intersects $V(P_\top)$ and $V(P_\bot)$ in a connected path in each case. Thus, we can place the first $b-m$ sources of fire by alternatingly placing vertices on the top and bottom path such that no vertex in $\mathcal{P}$ is burned by two sources, and that the vertices in $\mathcal{P}$ burned by these sources induce two paths. Furthermore, if we place the first source at $v_{b,1}$, and then continue placing these $b-m$ sources further to the right, we have that each source that burns for $i\geq m$ rounds burns a total of $4i-2m$ vertices ($2i-1$ vertices on one horizontal path and $2i-2m+1$ vertices on the other). These sources burn a total of
\[
\sum_{i=m}^b (4i-2m)=2b(b-m+1)=k
\]
vertices in $\mathcal{P}$. Furthermore, the vertices in $\mathcal{P}$ that are not burnt by these $b-m$ sources constitute at most three paths (a path containing the vertex $v_{m,1}$, another containing $v_{m,n}$, and a third containing $v_{1,n}$). The orders of these three paths sum up to $2n-k$, and if we consider these three paths as a spanning subgraph of $P_{2n-k}$, then via Lemma \ref{lemma subgraph burning} (applied with $X=V(P_{2n-k})$), the remaining vertices can be burned with at most $b(P_{2n-k})+2\leq \sqrt{2n-k}+2$ sources of fire. Thus, if $\sqrt{2n-k}+2\leq m$, we are done. Indeed, it can be routinely verified (for example, via SageMath) that as long as
\begin{align*}
b&\geq \frac{1}2\left(m+\sqrt{4n-m^{2}-2m+1}+1\right)\\
&=(1+o(1))\frac{1}2\left(m+\sqrt{4n-m^{2}}\right)\\
&=(1+o(1))\left(\frac{c}2+\sqrt{1-\frac{c^2}4}\right)\sqrt{n},
\end{align*}
we have that $\sqrt{2n-k}+2\leq m$, and we are done.
\end{proof}

Our final theorem establishes the upper bound when $c$ is small.

\begin{theorem}\label{proposition upper bound}
If $0< c\leq 2\sqrt{2}$, we then have that
\[
b(G_{c\sqrt{n},n})\leq (1+o(1))\left(\frac{c}2+\sqrt{1+\frac{c^2}{16}}\right)\sqrt{n}.
\]
\end{theorem}

\begin{proof}
	Let $P^{(1)}$ and $P^{(2)}$ be the horizontal paths at height $c/4$ and $3c/4-1,$ respectively, in $G_{c\sqrt{n},n}$, and let $\mathcal{P}=V(P^{(1)})\cup V(P^{(2)})$. Let $H\cong G_{\frac{c}2\sqrt{n}}$ be the subgraph of $G$ induced on the vertices at heights inclusively between $c/4$ and $3c/4-1$. We then have that
	\begin{align*}
	b(G_{c\sqrt{n},n}, \mathcal{P})&\leq b(H,\mathcal{P})\\
	&\leq (1+o(1))\left(\frac{c}4+\sqrt{1+\frac{c^2}{16}}\right),
	\end{align*}
	where the first inequality follows from Lemma \ref{lemma subgraph burning}, and the second follows from Lemma \ref{lemma burn top and bottom} applied with $c'=c/2\leq \sqrt{2}$. Since every vertex in $G_{c\sqrt{n},n}$ is distance at most $c/4+1$ from $P^{(1)}$ or $P^{(2)}$, via Lemma \ref{lemma product bound generalization}, we have that
	\begin{align*}
	b(G_{c\sqrt{n},n})&\leq b(G_{c\sqrt{n},n}, \mathcal{P})+ c/4 +1\\
	&\leq (1+o(1))\left(\frac{c}2+\sqrt{1+\frac{c^2}{16}}\right)\sqrt{n},
	\end{align*}
	and the proof follows.
\end{proof}

\section{Conclusions and Future Directions}

We found new bounds on the burning number of fence graphs $G_{c\sqrt{n},n}$ for constant $c$ in Theorem~\ref{thm main result}. We note that Theorem \ref{thm main result} implies that there exist constants $C_1$ and $C_2$ independent of both $n$ and $c$ such that
\[
C_1c^{1/3}\sqrt{n}\leq b(G_{c\sqrt{n},n})\leq C_2 c^{1/3}\sqrt{n},
\]
which is consistent with the growth rate of the bounds given in Theorem \ref{thm grid burning number} when $c=\omega(1)$. Our bounds in Theorem \ref{thm main result} are not asymptotically tight, so it would be interesting to determine the constant on the leading term for the burning number of fences.

Another direction worth exploring is extending our results to the setting of strong products. The \textit{strong product} of $G$ and $H$, denoted $G \boxtimes H$, has vertices $V(G) \times V(H)$, and edges $(u_1, u_2)(v_1, v_2) \in E$ if either $u_1 = v_1$ and $u_2 v_2 \in E(H)$, $u_1 v_1 \in E(G)$ and $u_2 = v_2$, or $u_1 v_1 \in E(G)$ and $u_2 v_2 \in E(H).$ Observe that $G \boxempty H$ is a spanning subgraphs of $G \boxtimes H$, so many of our results extend to strong products of paths. In \cite{Mits2018} it was found that for $m\leq n$,
\[
b\big(P_m\boxtimes P_n \big) = \begin{cases}
\big(1 + o(1) \big)\sqrt[3]{\frac{3}{4} mn} \: & \mbox{ if } m = \omega\big(\sqrt{n}\big), \\[0.5cm]
\Theta\big(\sqrt{n}\big) \: & \mbox{ if } m = O\big(\sqrt{n}\big).
\end{cases}
\]
Another open direction is to improve the bounds on $b(P_m\boxtimes P_n)$ in the case of $m = O(\sqrt{n})$.

\end{document}